\documentclass[a4paper, 12pt]{article}

\usepackage{amsmath, amssymb, amsthm}
\usepackage{pstricks}
\usepackage{authblk}

\def\margin{2.8cm}
\usepackage[left=\margin,right=\margin,top=\margin,bottom=\margin]{geometry}

\title{On Rainbow Connection Number and Connectivity}
\author{L.~Sunil~Chandran}
\author{Rogers~Mathew}
\author{Deepak~Rajendraprasad}
\affil
{
	Department of Computer Science and Automation, \authorcr 
	Indian Institute of Science, \authorcr
	Bangalore -560012, India. \authorcr
	\{sunil, rogers, deepakr\}@csa.iisc.ernet.in
}
\date{}

\theoremstyle{definition}
\newtheorem{defn}{Definition}

\theoremstyle{plain}
\newtheorem{thm}{Theorem}
\newtheorem{lem}[thm]{Lemma}
\newtheorem{cor}[thm]{Corollary}
\newtheorem{conj}[thm]{Conjecture}
\newtheorem{claim}{Claim}

\theoremstyle{remark}

\newtheorem{example}[thm]{Example}

\def\is{\leftarrow}
\def\NlBar{\overline{N^l}}
\def\-{\mbox{--}}

\newcommand{\halfc}[1]{\Big\lceil  \frac{#1}{2} \Big\rceil} 

\begin{document}

\maketitle

\begin{abstract}
{\em Rainbow connection number}, $rc(G)$, of a connected graph $G$ is the minimum number of colours needed to colour its edges, so that every pair of vertices is connected by at least one path in which no two edges are coloured the same. In this paper we investigate the relationship of rainbow connection number with vertex and edge connectivity. It is already known that for a connected graph with minimum degree $\delta$, the rainbow connection number is upper bounded by $3n/(\delta + 1) + 3$ [Chandran et al., 2010]. This directly gives an upper bound of $3n/(\lambda + 1) + 3$ and $3n/(\kappa + 1) + 3$ for rainbow connection number where $\lambda$ and $\kappa$, respectively, denote the edge and vertex connectivity of the graph. We show that the above bound in terms of edge connectivity is tight up-to additive constants and show that the bound in terms of vertex connectivity can be improved to $(2 + \epsilon)n/\kappa + 23/ \epsilon^2$, for any $\epsilon > 0$. We conjecture that rainbow connection number is upper bounded by $n/\kappa + O(1)$ and show that it is true for $\kappa = 2$. We also show that the conjecture is true for chordal graphs and graphs of girth at least $7$.
\end{abstract}


\section{Introduction}

{\em Edge colouring} of a graph is a function from its edge set to the set of natural numbers. A path in an edge coloured graph with no two edges sharing the same colour is called a {\em rainbow path}. An edge coloured graph is said to be {\em rainbow connected} if every pair of vertices is connected by at least one rainbow path. Such a colouring is called a {\em rainbow colouring} of the graph. The minimum number of colours required to rainbow colour a connected graph is called its {\em rainbow connection number}, denoted by $rc(G)$. For example, the rainbow connection number of a complete graph is $1$, that of a path is its length, and that of a star is its number of leaves. For a basic introduction to the topic, see Chapter $11$ in \cite{chartrand2008chromatic}.

The concept of rainbow colouring was introduced by Chartrand, Johns, McKeon and Zhang in 2008 \cite{chartrand2008rainbow}. Chakraborty et al. showed that computing the rainbow connection number of a graph is NP-Hard \cite{chakraborty2009hardness}. To rainbow colour a graph, it is enough to ensure that every edge of some spanning tree in the graph gets a distinct colour. Hence order of the graph minus one is an upper bound for rainbow connection number. Many authors view rainbow connectivity as one `quantifiable' way of strengthening the connectivity property of a graph \cite{caro2008rainbow, chakraborty2009hardness, krivelevich2010rainbow}. Hence tighter upper bounds on rainbow connection number for graphs with higher connectivity have been a subject of investigation. 

The following are the results in this direction reported in literature: Let $G$ be a graph of order $n$. Caro et al. have shown that if $G$ is 2-edge-connected (bridgeless), then $rc(G) \leq 4n/5 -1$ and if $G$ is 2-vertex-connected, then $rc(G) \leq \min\{2n/3, n/2 + O(\sqrt{n})\}$ \cite{caro2008rainbow}. Li and Shi have shown that if $G$ is $3$-vertex-connected, then $rc(G) \leq 3(n+1)/5$ \cite{li2010rain3con}. Krivelevich and Yuster showed that $rc(G) \leq 20n/\delta$, where $\delta$ is the minimum degree of $G$ \cite{krivelevich2010rainbow}. The result was recently improved by Chandran et al. where it was shown that $rc(G) \leq 3n/(\delta + 1) + 3$ \cite{chandran2010raindom}. Hence it follows that $rc(G) \leq 3n/(\lambda + 1) + 3$ if $G$ is $\lambda$-edge-connected and $rc(G) \leq 3n/(\kappa + 1) + 3$ if $G$ is $\kappa$-vertex-connected. This is because $\kappa \leq \lambda \leq \delta$.

In this paper we show that the bound of $3n/(\lambda + 1) + 3$ in terms of edge-connectivity is tight up to additive constants for infinitely many values of $\lambda$ and $n$ (Example \ref{ex:edgetight}). We improve the bound for $\kappa$-vertex-connected graphs to $rc(G) \leq (2 + \epsilon)n/ \kappa + 23/ \epsilon^2$ for any $\epsilon > 0$ (Corollary \ref{cor:raincon}). We conjecture (Conjecture \ref{conj:raincon}) that for $\kappa$-vertex-connected graphs, $rc(G) \leq n/ \kappa + O(1)$ and show that it is true for $\kappa = 2$ (Theorem \ref{thm:2con}). This improves the previous best known upper bounds for $2$-connected and $3$-connected graphs \cite{caro2008rainbow, li2010rain3con} mentioned in last paragraph. For $\kappa \geq 3$, we show that the conjecture is true for chordal graphs (Theorem \ref{thm:chordal}) and graphs of girth at least $7$ (Corollary \ref{cor:girthcon}). It can be easily shown from existing literature that the conjecture is true for all $\kappa$ for some other graph classes like AT-free graphs and circular arc graphs too \cite{chandran2010raindom}. We remark that an upper bound of $n/ \kappa + O(1)$ will be tight upto additive factors.

Another important recent development in this direction was by Basavaraju et al. \cite{basavaraju2010rainbow}, who showed that for any bridgeless graph $G$, $rc(G) \leq r(r+2)$ where $r$ is radius of $G$. This automatically renders the above conjecture true for graphs with radius at most $\sqrt{n/ \kappa} - 1$.  

\subsection{Preliminaries}
\label{sec:prelims}

All graphs considered in this article are finite, simple and undirected. The {\em length} of a path is its number of edges. If $S$ is a subset of vertices of a graph $G$, the subgraph of $G$ induced by the vertices in $S$ is denoted by $G[S]$. The vertex set and edge set of $G$ are denoted by $V(G)$ and $E(G)$ respectively. The order of $G$ (number of vertices) may be denoted by $|G|$.

\begin{defn}
{\em Vertex (edge) connectivity} of a graph is the minimum number of vertices (edges) whose removal disconnects the graph. A graph is called $\kappa$-vertex connected ($\lambda$-edge connected) if its vertex (edge) connectivity is at least $\kappa$ ($\lambda$). A $2$-edge-connected graph is called {\em bridgeless}. A $\kappa$-vertex-connected graph may also be referred to as a $\kappa$-connected graph.
\end{defn}


\begin{defn}
Let $G$ be a connected graph. The {\em distance} between two vertices $u$ and $v$ in $G$, denoted by $d_G(u,v)$, is the length of a shortest path between them in $G$. The {\em eccentricity} of a vertex $v$ is $ecc(v) := \max_{x \in V(G)}{d_G(v, x)}$. The {\em diameter} of $G$ is $diam(G) := \max_{x \in V(G)}{ecc(x)}$. The {\em radius} of $G$ is $rad(G) := \min_{x \in V(G)}{ecc(x)}$. Distance between a vertex $v$ and a set $S \subseteq V(G)$ is $d_G(v, S) := \min_{x \in S}{d_G(v,x)}$. The neighbourhood of $S$ is $N_G(S) := \{x \in V(G) | d_G(x, S) =1\}$.
\end{defn}

\begin{defn}
\label{defn:domination}
Given a graph $G$, a set $D \subseteq V(G)$ is called a {\em $l$-step dominating set} of $G$, if every vertex in $G$ is at a distance at most $l$ from $D$. Further if $G[D]$ is connected, then $D$ is called a {\em connected $l$-step dominating set} of $G$. 
\end{defn}

\begin{lem}[\cite{basavaraju2010rainbow}]
\label{lem:domgrow}
If $G$ is a bridgeless graph, then for every connected $l$-step dominating set $D^l$ of $G$, $l \geq 1$, there exists a connected $(l-1)$-step dominating set $D^{l-1} \supset D^l$ such that 
$$rc(G[D^{l-1}]) \leq rc(G[D^l]) + 2l + 1.$$
\end{lem}

\begin{defn}
A graph is called {\em chordal} if it contains no induced cycles of length greater than $3$. 
\end{defn}

\begin{defn}
{\em Girth} of a graph is the size of a shortest cycle in the graph. It will be denoted by $girth(G)$.
\end{defn}

\section{Our Results}

First we settle the question of a tight upper bound for rainbow connection number in terms of edge-connectivity by showing that the bound of $3n/(\lambda + 1) + 3$ which directly follows from the minimum degree bound of $3n/(\delta + 1) + 3$ \cite{chandran2010raindom} is tight upto additive factors. We show the tightness by constructing a family of $\lambda$-edge-connected graphs for infinitely many values of $\lambda$ and order $n$ with diameter $d = \frac{3n}{\lambda+1} -3$. Since diameter is a lower bound on the rainbow connection number, the construction suffices for our purpose. 

\begin{example}[Construction of a $\lambda$-edge connected graph $G$ on $n$ vertices with diameter $d = \frac{3n}{\lambda + 1} - 3$] 
\label{ex:edgetight}
Let $d \geq 1$ be a natural number. Let $\lambda$ be a natural number such that $\lambda + 1$ is a multiple of $3$ and $\lambda \geq 8$. Let $k := \frac{\lambda + 1}{3}$. Let $V(G) = V_0 \uplus V_1 \uplus \cdots \uplus V_d$, where $|V_i|$ is $2k$ for $i=1$ and $i=d$ and $k$ for $1 < i < d$. Two distinct vertices $u \in V_i$ and $v \in V_j$ are adjacent in $G$ iff $|i-j| \leq 1$. It is easy to see that diameter of $G$ is $d$, $n = |V(G)| = k(d+3)$ and hence $d = \frac{n}{k} - 3 = \frac{3n}{\lambda + 1} - 3$. By considering all pairs of vertices, it can be seen that $G$ is $\lambda$-edge-connected. 
\end{example}

Next, we try to obtain an upper bound on rainbow connection number for $\kappa$-vertex-connected graphs that is tighter than the $3n/(\kappa+1)$ bound implied by the degree bound in \cite{chandran2010raindom}. We will show that for any $\kappa$-vertex-connected graph $G$, $rc(G) \leq (2 + \epsilon)n/ \kappa + 23/ \epsilon^2$ for any $\epsilon > 0$. We prove the result after stating and proving the following two lemmas. 

\begin{lem}
\label{lem:domrain}
If $G$ is a bridgeless graph, and $D^l$ is a connected $l$-step dominating set of $G$, then
$$rc(G) \leq rc(G[D^l]) + l(l+2) \leq |D^l| -1 + l(l+2).$$
\end{lem}
\begin{proof}
Note that the only $0$-step dominating set in $G$ is $V(G)$. Hence the first inequality follows from repeated application of Lemma \ref{lem:domgrow}. The second inequality follows since $rc(G[D^l]) \leq |D^l|-1$.
\end{proof}

\begin{lem}
\label{lem:domcon}
Every $\kappa$-vertex-connected ($\kappa \geq 1$) graph $G$ of order $n$ has a connected $2l$-step dominating set of size at most $\big(\frac{2l+1}{\kappa l + 1}\big)n$ for every natural number $l \geq 0$.
\end{lem}
\begin{proof}
If $k \leq 2$, the bound is trivial for any $l \geq 0$ since we can take $V(G)$ as the dominating set. Similarly if $r$ is the radius of $G$, for $l \geq r/2$ we can take any central vertex of $G$ as the $2l$-step dominating set. Hence we assume $\kappa > 2$ and $l < r/2$.

The following procedure is used to construct a $2l$-step dominating set $D$. Let $N^{i}(S) := \{x : d_G(x,S) = i\}$ and $\overline{N^i}(S) :=  \{x : d_G(x, S) \leq i \}$ for any $S \subset V(G)$. $N^i(s) = N^i(\{s\})$ and $\overline{N^i}(s) = \overline{N^i}(\{s\})$ for any $s \in V(G)$.
\begin{quote}
	$D = \{u\}$, for some $u \in V(G)$.\\ 
	While $N^{2l+1}(D) \neq \emptyset$,\\
	\{\\
	\hspace*{5ex}Pick any $v \in N^{2l+1}(D)$. Let $(v, x_{2l}, x_{2l-1}, \ldots, x_0)$, $x_0 \in D$ be a shortest $v\-D$ path.\\
	\hspace*{5ex}$D \is D \cup \{x_1, x_2, \ldots, x_{2l}, v \}$. \\ 
	\}
\end{quote}

Clearly $D$ remains connected after every iteration of the procedure. Since the procedure ends only when $N^{2l+1}(D) = \emptyset$, the final $D$ is a $2l$-step dominating set. Let $t$ be the number of iterations executed by the procedure. When the procedure starts $|\NlBar(D)| = |\NlBar(u)|\geq \kappa l + 1$. This is because $l < r$ and $|N^i(u)| \geq \kappa$ for $1 \leq i  < r$ since $G$ is $\kappa$-connected. Note that $v \in N^{2l+1}(D)$ ensures that $\NlBar(v) \cap \NlBar(D) = \emptyset$, and $|\NlBar(v)| \geq \kappa l + 1$ due to $\kappa$-connectivity of $G$. Hence the addition of $v$ to $D$ increases $|\NlBar(D)|$ by at least $\kappa l + 1$ in every iteration. Therefore, when the procedure ends, $(\kappa l + 1)(t + 1) \leq |\NlBar(D)| \leq n$. Since $D$ starts as a singleton set and each iteration adds $2l+1$ more vertices, $|D| = (2l+1)t + 1 \leq  \frac{(2l+1)n}{\kappa l + 1} - 2l \leq \big(\frac{2l+1}{\kappa l + 1}\big)n$.
\end{proof}

\begin{thm}
\label{thm:raincon}
If $G$ is a $\kappa$-vertex-connected ($\kappa \geq 1$) graph of order $n$, then for every natural number $l \geq 0$, 
$$rc(G) \leq \left(\frac{2l+1}{\kappa l + 1}\right)n + 2l(2l + 2)-1.$$
\end{thm}
\begin{proof}
The case $\kappa=1$ is trivial. Hence we assume $\kappa \geq 2$ and therefore $G$ is bridgeless. 
Since $G$ is $\kappa$-connected, by Lemma \ref{lem:domcon}, for every $l \geq 0$ we have a $2l$-step dominating set $D$ of size at most $\big(\frac{2l+1}{\kappa l + 1}\big)n $. Now an application of Lemma \ref{lem:domrain} gives the result.
\end{proof}

\begin{cor}
\label{cor:raincon}
For every $\kappa \geq 1$, if $G$ is a $\kappa$-vertex-connected graph of order $n$, then for every $\epsilon \in (0,1)$, 
$$rc(G) \leq \left(\frac{2 + \epsilon}{\kappa}\right) n + \frac{23}{\epsilon^2}.$$
\end{cor}
\begin{proof}
Given an $\epsilon \in (0,1)$, choose $l = \lceil \frac{1}{\epsilon}\rceil$. Then the result follows from Theorem \ref{thm:raincon}. Note that $2l(2l + 2) -1 \leq  23/\epsilon^2$.
\end{proof}

The above bound may not be tight, and we are tempted to believe that the following conjecture might be true.

\begin{conj}
\label{conj:raincon}
For every $\kappa \geq 1$, if $G$ is a $\kappa$-vertex-connected graph of order $n$, then $rc(G) \leq n / \kappa + O(1).$
\end{conj}

Now we show some cases in which Conjecture \ref{conj:raincon} is true, namely high girth graphs, chordal graphs and graphs of vertex-connectivity $2$.

\begin{lem}
\label{lem:domgirth}
Every connected graph $G$ of order $n$, minimum degree $\delta \geq 3$ and girth at least $2g + 1$ has a connected $2g$-step dominating set of size at most $\big(\frac{2g+1}{C_{\delta,g}}\big)n - 2g$, where $C_{\delta,g} = \frac{\delta(\delta-1)^g - 2}{\delta-2}$.
\end{lem}
\begin{proof}
Note that $\delta \geq 3$, ensures that $G$ is not a tree and hence the girth is finite. That is $1 \leq g < \infty$. Now observe that for any vertex $v \in V(G)$,  $|\NlBar(\{v\})| \geq 1 + \delta + \delta(\delta - 1) + \ldots + \delta(\delta - 1)^{g-1}$. The summation on the right hand side is equal to $C_{\delta,g}$, for $\delta \geq 3$. Now the proof follows the same steps as in that of Lemma \ref{lem:domcon} after setting $l = g$. Hence we omit the details. 
\end{proof}

\begin{cor}
\label{cor:girthdegree}
If $G$ is a connected graph of minimum degree $\delta$ then
\begin{enumerate}
\item if $\delta \geq 3$ and $girth(G) \geq 7$, then $rc(G) \leq n/\delta + 41$ and 
\item if $\delta \geq 5$ and $girth(G) \geq 5$, then $rc(G) \leq n/\delta + 19$.
\end{enumerate}
\end{cor}
\begin{proof}
The proof follows from substituting the mentioned values of minimum degree and girth in Lemma \ref{lem:domgirth} and then by applying Lemma \ref{lem:domrain}.
\end{proof}

Since vertex connectivity of a graph is a lower bound for minimum degree, the following corollary is immediate.

\begin{cor}
\label{cor:girthcon}
If $G$ is a $\kappa$-vertex-connected graph such that $\kappa \geq 3$ and $girth(G) \geq 7$, then $rc(G) \leq n/\kappa + 41$.
\end{cor}

\begin{thm}
\label{thm:chordal}
For every $\kappa$-vertex-connected chordal graph $G$ of order $n$,
$$ rc(G) \leq \frac{n}{\kappa} + 3.$$
\end{thm}
\begin{proof}
The case of $\kappa = 1$ is trivial since rainbow colouring a spanning tree of $G$ suffices. Hence let us assume $\kappa \geq 2$ and hence $G$ is bridgeless. We claim that $G$ has a $1$-step connected dominating set $D$ which can be rainbow coloured using $\frac{|D|}{\kappa}$ colours. Then by Lemma \ref{lem:domrain}, $rc(G) -3 \leq rc(D) \leq \frac{|D|}{\kappa} \leq \frac{n}{\kappa}$. Hence it remains to prove the  above claim. Consider a maximal connected set $D \subset V(G)$ that can be rainbow coloured using $\frac{|D|}{\kappa}$ colours. Such a set exists since any singleton set of vertices can be rainbow coloured using $0 < \frac{1}{\kappa}$ colours. Suppose for contradiction that $D$ is not a $1$-step dominating set. Then $N_G(D)$ is a vertex separator and hence contains a minimal separator $S \subset N_G(D)$. Since $G$ is $\kappa$-connected, $|S| \geq \kappa$, and since $G$ is chordal, $S$ induces a clique \ref{}[CHORDAL PROPERTY?]. Giving a single new colour to every $D\-S$ and $S\-S$ edge extends the rainbow colouring of $G[D]$ to $G[D \cup S]$. Thus $D \cup S$ is a connected set containing $D$ which can be coloured using $rc(G[D]) + 1 \leq \frac{|D|}{\kappa} + 1 \leq \frac{|D \cup S|}{\kappa}$ colours contradicting the maximality of $D$. So $D$ is a $1$-step connected dominating set and thus the result follows.
\end{proof}

\subsection{Rainbow colouring of $2$-vertex-connected Graphs}

The last result to be proved in this article is that $rc(G) \leq \halfc{n} + 1$ for every two-connected graph $G$. In order to prove it, we introduce some terms and notations and then state and prove a few lemmas.

Let $H$ be a connected graph and $P$ be an (open) ear on $H$, i.e., $P$ is a path $(x_0, x_1, \ldots, x_m)$ such that $V(P) \cap V(H) = \{x_0, x_m\}$ and $x_0 \neq x_m$. The number of edges in $P$ is called its {\em length}. An ear is termed {\em even} or {\em odd} based on its length. The vertices $x_0$ and $x_m$ are called the {\em foots} of $P$ and the remaining vertices of $P$ are called {\em internal vertices}. If $P$ is an even ear, then $x_{m/2}$ is called its {\em tip} and will be denoted by $tip(P)$. An internal vertex of an ear which is not a tip is called a {\em balanced vertex}. Note that the number of new vertices introduced by an ear is one less that its length. A path $P$ in a $2$-connected graph $G$ is called a {\em removable ear} of $G$ if $G$ remains $2$-connected after the removal of internal vertices of $P$. Further , $P$ is called a {\em clean-removable ear} if $P$ is removable and every internal vertex of $P$ has degree exactly $2$.

Now we introduce two colouring schemes to extend a rainbow colouring of $H$ to $H \cup P$. Given an edge coloured graph $G$, we use $colours(G)$ to denote the set of colours used in the colouring of $G$. 

\begin{defn}[Balanced colouring of an odd ear]
Let $P = (x_0, x_1, \ldots, x_{2k+1})$ be an odd ear on a rainbow coloured graph $H$. We define a colouring $C$ of the edges of $P$ as follows. $C(x_j,x_{j+1}) = C(x_{j+k+1},x_{j+k+2}) = c_j$ for $0 \leq j \leq k-1$ and $C(x_k, x_{k+1} ) = c_{old}$, where $c_j \notin colours(H)$ and $c_{old} \in colours(H)$. The exact colour to be used as $c_{old}$ may be specified if necessary when this procedure is invoked. Then $C$ is called a {\em balanced colouring of $P$}, and every edge of $P$ which gets one of the colours in $\{c_0, c_1, \ldots, c_{k-1} \}$ is called a {\em balanced edge}.
\end{defn}

\begin{lem}
\label{lem:oddear}
Let $H$ be a rainbow coloured graph which is coloured using at most $h$ colours, and $P$ an odd ear on $H$ which is assigned a balanced colouring. Then $G := H \cup P$ is rainbow coloured. Further, $rc(H \cup P) \leq h + \frac{m}{2}$ where $m = |V(H \cup P)| - |V(H)|$ is the number of new vertices added by $P$.
\end{lem}
\begin{proof}
We prove that $G$ is rainbow coloured by demonstrating a rainbow path between every pair of vertices in $G$. Let $P = (x_0, x_1, \ldots, x_{2k+1})$. Since $H$ is rainbow coloured, every pair vertices in $H$ is already connected by a rainbow path in $H$. If one vertex $v$ in the pair is in $H$ and the other is an internal vertex $x_i$ of $P$, then $x_i \-f$ path in $P$ joined with the $f\-v$ rainbow path in $H$, where $f$ is the foot of $P$ nearer to $x_i$, is an $x_i\-v$ rainbow path. If both the vertices in the pair, $x_i$ and $x_j$, $i < j$ are in $P$, then if $j-i \leq k+1$, the portion of $P$ between these two vertices is a rainbow path. Otherwise, we get a rainbow path between them by joining the remaining portion of $P$ with a rainbow path in $H$ between the two foots of $P$. Note that the number of new vertices added by $P$ is $m = 2k$. It is evident from the balanced colouring procedure that $rc(H \cup P)$ is at most $h + k$.
\end{proof}

\begin{defn}[Balanced colouring of an even ear]
Let $P = (x_0, x_1, \ldots, x_{2k})$ be an even ear on a rainbow coloured graph $H$. We define a colouring $C$ of the edges of $P$ as follows. $C(x_j,x_{j+1}) = C(x_{j+k+1},x_{j+k+2}) = c_j$ for $0 \leq j \leq k-2$, $C(x_{k-1}, x_{k}) = t_1$ and $C(x_k, x_{k+1} ) = t_2$, where $c_j \notin colours(H)$, and $t_1, t_2 \notin \{c_0, c_1, \ldots, c_{k-2}\}$ will be two distinct colours. The exact colours to be used as $t_1$ and $t_2$ will be specified when this colouring procedure is invoked. Such a colouring is called a {\em balanced colouring} of $P$ and every edge of $P$ which gets one of the colours in $\{c_0, c_1, \ldots, c_{k-2} \}$ is called a {\em balanced edge}. Edges $(x_{k-1}, x_k)$ and $(x_k, x_{k+1})$ are called the first and second {\em tip edges} respectively, and $t_1$ and $t_2$ are called the {\em tip colours} of $P$.
\end{defn}

\begin{lem}
\label{lem:evenear}
Let $H$ be a rainbow coloured graph, and $P$ an even ear on $H$ which is assigned a balanced colouring. Let $G := H \cup P$. Then there exists a rainbow path between every pair of vertices in $G$ except in the case when one vertex in the pair is $tip(P)$ and the other is in $H$. Further, the number of new colours introduced by the balanced colouring of $P$ is $\frac{m-1}{2} + t$, where $m$ is the number of new vertices added by $P$ and $t = |\{t_1, t_2\} \setminus colours(H)|$ denotes the number of fresh colours among the tip colours.
\end{lem}
\begin{proof}
Let $P = (x_0, x_1, \ldots, x_{2k})$. Since $H$ is rainbow coloured, every pair vertices in $H$ is already connected by a rainbow path in $H$. If one vertex $v$ in the pair is in $H$ and the other is a balanced vertex $x_i$ of $P$, then $x_i \-f$ path in $P$ joined with the $f\-v$ rainbow path in $H$, where $f$ is the foot of $P$ nearer to $x_i$, is an $x_i\-v$ rainbow path. If both the vertices in the pair, $x_i$ and $x_j$, $i < j$ are in $P$, then if $j-i \leq k+1$, the portion of $P$ between these two vertices is a rainbow path. Otherwise, we get a rainbow path between them by joining the remaining portion of $P$ with a rainbow path in $H$ between the two foots of $P$. Note that the number of new vertices added by $P$ is $m = 2k-1$. It is evident from the balanced colouring procedure that the number of new colours used is at most $k-1 + t$.
\end{proof}

\begin{thm}
\label{thm:2con}
For every $2$-vertex-connected graph $G$ of order $n$,
$$ rc(G) \leq \halfc{n} + 1.$$
\end{thm}
\begin{proof}
Let $G$ be a minimal counter example to the above statement, i.e., $G$ is a $2$-connected graph with $rc(G) > \halfc{n} +1$ and every proper $2$-connected subgraph $H$ of $G$ has $rc(H) \leq \halfc{|H|} + 1$. First we argue that $G$ cannot have a removable odd ear. Let $P$ be a removable odd ear of $G$ and $G'$ be the subgraph obtained by removing the internal vertices of $P$ from $G$. Since $G'$ is a proper $2$-connected subgraph of $G$, it can be coloured using at most $\halfc{|G'|} + 1$ colours. This colouring can be extended to $G$ by giving a balanced colouring to $P$ (Lemma \ref{lem:oddear}). Hence $rc(G) \leq \halfc{n} + 1$ and therefore $G$ cannot be a counter example to the statement. 

Let $H$ be a maximal $2$-connected subgraph of $G$ that can be rainbow coloured using $\halfc{|H|}$ colours. Such an $H$ exists since any cycle in $G$ of $m$ vertices can be rainbow coloured using $\halfc{m}$ colours.  

If $V(H) = V(G)$, then $G$ is not a counter example. Hence there exists a vertex $v$ in $G$ outside $H$ and by $2$-connectivity of $G$, there exists two internally vertex disjoint paths from $v$ to $H$ whose union forms an ear on $H$. Let $P_1$ be one of the largest such ears. 

\begin{claim}
\label{claim1}
$P_1$ is a clean-removable even ear in $G$.
\end{claim}

If $P_1$ is an odd ear, we can extend the rainbow colouring of $H$ by a balanced colouring of $P_1$ so that the $2$-connected graph $H' := H \cup P_1$ is rainbow coloured using at most $\halfc{|H'|}$ colours (Lemma \ref{lem:oddear}). This contradicts the maximality of $H$. If $P_1$ is not a clean-removable ear, then we will show that there exists a $2$-connected proper super-graph $H'$ of $H$ which can be rainbow coloured using at most $\halfc{|H'|}$ colours thereby contradicting the maximality of $H$.

Given $H$ and the even ear $P_1$ which is not clean-removable, we construct $H'$ as follows. Let $H' := H \cup P_1 \cup P_2 \cup \cdots \cup P_{l+1}$, where $P_{l+1}$ is the first even ear in the sequence after $P_1$ and every $P_{i+1}$ is the largest ear in $G$ that can be added on $H_i := H \cup P_1 \cup P_2 \cup \cdots \cup P_i$ which has at least one foot as an internal vertex of $P_i$. We will always get an even ear $P_{l+1}$ because otherwise this construction ends at a clean-removable odd ear in $G$ which contradicts the minimality of $G$. 

We extend the rainbow colouring of $H$ to $H'$ as follows. Give a balanced colouring to the even ear $P_1$ using one new colour $t_{new}$ as the first tip edge colour $t_1$ and any colour $t_{old} \in colours(H)$ as the second tip edge colour $t_2$. We can find a rainbow path from $tip(P_1)$ to every vertex in $H$ through the first tip edge coloured $t_{new}$. Hence by Lemma \ref{lem:evenear}, we can see that $H_1$ is rainbow coloured and $rc(H_1) \leq \halfc{|H_1|} + 1$. If $l \geq 2$, we give a balanced colouring to every odd ear $P_2$ to $P_l$ ensuring that $t_{new}$ is not used as the old colour ($c_{old}$) for the middle edge of any of them. Therefore $H_l$ is rainbow coloured (Lemma \ref{lem:oddear}) and the colouring uses at most $\halfc{|H_l|} + 1$ colours. Let $f_1$ and $f_2$ be the foots of $P_{l+1}$, one of which is an internal vertex of $P_l$. Since we always add the largest possible ear incident on the previous ear in each step, $f_1$ and $f_2$ are non-adjacent in $H_l$. Let $R$ be a rainbow path in $H_l$ between $f_1$ and $f_2$. Now. $R \cap P_l$ may or may not include a balanced edge of $P_l$ based on which we branch into two cases.
  
In the former case, when there is a balanced edge of $P_l$ in $R \cap P_l$, let $c$ be the colour assigned to that edge in the balanced colouring of $P_l$. There are only two edges in $H_l$ with colour $c$ since $c$ is a new colour used in the balanced colouring of $P_l$ which is the last ear added in the construction of $H_l$. The two edges of $P_l$ coloured $c$ form an edge-cut of $H_l$ in such a way that one component is a subgraph of $P_l$ and the other contains $H_{l-1}$. Since there is no path between two vertices of the same component with exactly one $c$-coloured edge, we see that $f_1$ and $f_2$ are not in the same component. Let $C_1$ be the component containing $f_1$ and $C_2$ be the one containing $f_2$.   It is easy to see that both $C_1$ and $C_2$ are rainbow coloured without using colour $c$. Without loss of generality, let $tip(P_1) \in C_1$. This further ensures that $C_2$ does not use the colour $t_{new}$. Now we can give a balanced colouring of the even ear $P_{l+1} = (f_1 = x_0, x_1, \ldots, x_{2k} = f_2)$ with $t_1 = c$ and $t_2 = t_{new}$ as the tip colours. Recall that it is the edge $e_1 = (x_{k-1}, x_k)$ which gets the colour $t_1$. Since $C_1$ is rainbow coloured without using colour $c$, there is a rainbow path from $tip(P_{l+1})$ through $e_1$ to every vertex in $C_1$. Similarly, there is a rainbow path from $tip(P_{l+1})$ to every vertex in $C_2$ through $e_2 = (x_k, x_{k+1})$ which gets colour $t_{new}$. Hence by Lemma \ref{lem:evenear}, $H' = H_{l+1}$ is rainbow coloured. We have used only $k-1$ new colours to colour $P_{l+1}$ which has $2k-1$ new vertices and hence $rc(H') \leq \halfc{|H'|}$. 

The latter case i.e., the case where there is no balanced edge of $P_l$ in $R \cap P_l$, occurs only when $P_l$ is an even ear and the two foots of  $P_{l+1}$ are vertices that are adjacent to $tip(P_l)$ in $P_l$. Hence, first of all, $l=1$ as all but the first and last ears in the sequence are odd ears. Further since $P_1$ is the largest possible ear on $H$, and since both the foots of $P_2$ are only a distance $2$ apart in $P_1$, ear $P_2$ is of length $2$. Colour the two edges of $P_2$ so that the $4$-cycle formed by $P_2$ and $tip(P_1)$ is coloured alternatively with $t_{new}$ and $t_{old}$ where $t_{new}$ and $t_{old}$ are the two tip colours used during the balanced colouring of $P_1$. Every vertex on $H$ and vertices on one side of $P_1$ including $tip(P_1)$  can be reached from $tip(P_2)$ using rainbow paths through the edge of $P_2$ which is coloured $t_{new}$. Remaining vertices of $P_1$ can be reached from $tip(P_2)$ using rainbow paths through the other edge of $P_2$ which is coloured $t_{old}$. Hence by Lemma \ref{lem:evenear}, $H' = H \cup P_1 \cup P_2$ is rainbow coloured and we have used only $\halfc{|H'|}$ colours in total for the same.

This completes a proof of Claim \ref{claim1}.

\begin{claim}
\label{claim2}
$G = H \cup Q_1 \cup Q_2 \cup \cdots \cup Q_k$, where every $Q_i$ is a clean-removable even ear on $H$ and $k \leq 3$
\end{claim}

Let $Q_2 = P_1$ and $Q_2$ be one of the largest ears on $H$ in $G \setminus Q_1$, if it exists. Similar arguments as in the proof of Claim \ref{claim1} shows that, $Q_2$ is also a clean-removable even ear. Extend this construction so that $Q_{i+1}$ is a largest possible ear on $H$ in $G \setminus (Q_1 \cup Q_2 \cup \cdots \cup Q_i)$, if it exists. Hence $G = H \cup Q_1 \cup Q_2 \cup \cdots \cup Q_k$, where every $Q_i$ is a clean-removable even ear on $H$. 

If $k \geq 4$, then we show that $H' = H \cup Q_1 \cup Q_2 \cup Q_3 \cup Q_4$ can be rainbow coloured using at most $\halfc{|H'|}$ colours thereby contradicting the maximality of $H$. Give separate balanced colouring to all the four even ears $Q_1, \ldots, Q_4$, but with the same set of tip colours $t_1$ and $t_2$ which are two new colours not used in the colouring of $H$ or in the balanced colouring of any of the above four ears. Observe that from the tip of every such ear added, there are two rainbow paths to every vertex in $H$, one which avoids colour $t_1$ and the other which avoids colour $t_2$. Now we can see, by analysing different kinds of pairs and using Lemma \ref{lem:evenear}, that every pair of vertics in $H'$ is connected by a rainbow path. Further, we have used only $m/2$ extra colours, where $m$ is the total number of new vertices added to $H$ by the four ears. Hence $rc(H') \leq \halfc{|H'|}$ and so $k \leq 3$.


This proves Claim \ref{claim2}.

Hence $G = H \cup Q_1 \cup Q_2 \cup \cdots \cup Q_k$, where every $Q_i$ is a clean-removable even ear on $H$ with $k \in \{1, 2, 3\}$ and $rc(H) \leq \halfc{|H|}$.  In the case when $k=1$, we extend a minimum rainbow colouring of $H$ by a giving a balanced colouring to the even ear $Q_1$ with a new colour $t_{new} \notin colours(H)$ as first tip colour $t_1$ and an old colour $t_{old} \in colours(H)$ as the second tip colour $t_2$. In cases $k=2$ and $k=3$, we extend a minimum rainbow colouring of $H$ to $G$ as done in the proof of Claim \ref{claim2} above. That is, every $Q_i$ is given a separate balanced colouring but with the same set of tip colours not used in the colouring of $H$ or in the balanced colouring of any of $Q_i$. Now, using Lemma \ref{lem:evenear}, it is easy to verify that $G$ becomes rainbow coloured and the total number of colours used is at most $\halfc{|G|} + 1$. Hence $G$ cannot be a counter example to the original statement. This completes the proof of Theorem \ref{thm:2con}.

\end{proof}

\bibliographystyle{plain}
\bibliography{rainbow}
\end{document}